\documentclass{amsart}
\usepackage{tikz-cd}
\usepackage{centernot}
\usepackage{amssymb}
\usepackage{multirow}

\newcommand\BIT{\begin{enumerate}}
\newcommand\EIT{\end{enumerate}}

\newtheorem{thm}{Theorem}[section] 

\newtheorem{defn}[thm]{Definition}

\newtheorem{lem}[thm]{Lemma}
\newtheorem*{lem*}{Lemma}

\newtheorem{ques}[thm]{Question}

\newcommand{\Cref}[1]{{Corollary~\ref{#1}}}

\DeclareMathOperator{\rank}{\mathrm{rk}}
\DeclareMathOperator{\Span}{\mathrm{Span}}
\DeclareMathOperator{\Image}{\mathrm{Im}}
\DeclareMathOperator{\Ker}{\mathrm{Ker}}

\long\def\forget#1\forgotten{{}}

\title{Stable finiteness does not imply linear soficity}

\author{Be'eri Greenfeld}

\thanks{The author thanks L.~Small for related discussions on PI-algebras.}

\begin{document}
\begin{abstract}
We prove that there exist finitely generated, stably finite algebras which are non linear sofic. This was left open by Arzhantseva and Păunescu in 2017.
\end{abstract}

\maketitle

\section{Introduction}

One of the most tantalizing open problems in algebra is whether a non-sofic group exists. Recall that a group is \textit{sofic} if it can be approximated by almost homomorphisms to symmetric groups, equivalently, if it embeds into a metric ultraproduct of finite symmetric groups endowed with the normalized Hamming distance. 
Other important variants of soficity include \textit{hyperlinearity}, in which the symmetric groups are replaced by unitary groups, endowed with the normalized Hilbert-Schmidt norm (this is closely related to Connes' Embedding Conjecture), and \textit{linear soficity}, in which one considers metric ultraproducts of general linear groups endowed with the normalized rank function. Every sofic group is both hyperlinear \cite{EleSzaMA} and linear sofic \cite{ArzPau}. See also \cite{Pestov}.

Arzhantseva and Păunescu \cite{ArzPau} studied linear soficity of groups and algebras, and proved that a group $G$ is linear sofic if and only if its group algebra $\mathbb{C}[G]$ is linear sofic. Let us recall the required definition from \cite{ArzPau}.

Let $\mathcal{U}$ be a non-principal ultrafilter on the natural numbers and $(n_k)_{k}$ a sequence of natural numbers tending to infinity. We define the asymptotic rank function: $$ \rho_\mathcal{U}\colon \prod_{k} M_{n_k}(F)\rightarrow~[0,1]\ \ \ \ \ \text{by}\ \ \rho_\mathcal{U}\left((A_k)_k\right):=\lim_{k \rightarrow \mathcal{U}} \frac{1}{n_k}\rank(A_k) $$
Then one can form the metric ultraproduct $\prod_{k\rightarrow \mathcal{U}} M_{n_k}(F) / \Ker(\rho_\mathcal{U})$.

\begin{defn}[{\cite{ArzPau}}]
A countably generated algebra $A$ over a field $F$ is linear sofic if there
exists an injective homomorphism $\Phi\colon A\rightarrow \prod_{k\rightarrow \mathcal{U}} M_{n_k}(F) / \Ker(\rho_\mathcal{U})$. 
\end{defn}

While no examples of non linear sofic groups are known, it is not hard to find examples of non linear sofic algebras, based on the following observation. We say that a unital ring $A$ is \textit{directly finite}\footnote{Aka `Dedekind-finite' or `Von Neumann finite'.} if $xy=1$ implies $yx=1$ for every $x,y\in A$, and \textit{stably finite} if $M_n(A)$ is directly finite for every $n\in \mathbb{N}$. There exist examples of directly finite but non-stably finite rings \cite{She}.
It is straightforward to check that any metric ultraproduct $\prod_{k\rightarrow \mathcal{U}} M_{n_k}(F) / \Ker(\rho_\mathcal{U})$ is stably finite, hence every linear sofic algebra is. The algebra $F\left<x,y\right>/\left<xy-1\right>$ is non-directly finite, hence non linear sofic.
An open conjecture of Kaplansky asserts that the group algebra of an arbitrary group is directly finite (Kaplansky proved it for fields of characteristic zero); Elek and Szabó proved Kaplansky's conjecture for sofic groups \cite{EleSza} (a different proof is given in \cite{ArzPau}).

In \cite{ArzPau}, the authors mention that ``Such [stably finite non linear sofic] algebras seem difficult to find as counterexamples to soficity in general proved to be elusive." The aim of this note is to prove that stably finite, non linear sofic algebras exist.
Our proof is based on an asymptotic linear algebraic analysis of certain non-commutative equations, which we then show that can be solved in various stably finite algebras. The first instance is obtained using an example of Irving \cite{Irving}, from the theory of polynomial identity (PI) algebras:

\begin{thm} \label{main}
Over an arbitrary field, there exists a finitely generated non linear sofic algebra which satisfies a polynomial identity and is thus stably finite.
\end{thm}

Another example, of a completely different flavor, arises from the Cohn-Sasiada construction of a simple Jacobson radical ring \cite{CohSas}. 

\begin{thm} \label{main_jac}
Over an arbitrary field, there exists a finitely generated non linear sofic algebra which is Jacobson radical and whose unital hull\footnote{The unital hull of an $F$-algebra $R$ is the vector space $R^1:=F\oplus R$ with multiplication: \\ $(\alpha+r)\cdot (\alpha'+r'):=\alpha\alpha'+\alpha r'+\alpha' r+rr'$.} is thus a non linear sofic stably finite algebra.
\end{thm}

\smallskip
\textit{Conventions}. Throughout, $F$ is an arbitrary field; algebras are associative but not necessarily commutative; for a matrix $P\in M_n(F)$ we let $\Image(P)$ denote the image/column space of $P$, and let $\rank(P)=\dim_F \Image(P)$ denote its rank.

\section{non linear soficity}

In this section we prove the following non-soficity machinery:

\begin{lem} \label{main_lem}
Let $A$ be an $F$-algebra containing non-zero elements $x,y,z\in A$ such that:
\begin{itemize}
    \item $x\in yxA$
    \item $z\in xA$ and $yz=0$
\end{itemize}
Then $A$ is non linear sofic.
\end{lem}

\begin{proof}
On the contrary, if $A$ is linear sofic then we have an embedding: $$ \Phi\colon A \rightarrow \prod_{k\rightarrow \mathcal{U}} M_{n_k}(F) / \Ker(\rho_\mathcal{U}). $$ Fix a linear lift of $\Phi$ to $\prod_k M_{n_k}(F)$, say, $\widehat{\varphi}\colon A \rightarrow \prod_k M_{n_k}(F)$ so: $$ \Phi(a)=0 \iff \widehat{\varphi}\in \Ker(\rho_\mathcal{U}). $$ Write $\widehat{\varphi}=\prod_k \varphi_k$ with each $\varphi_k\colon A\rightarrow M_{n_k}(F)$. For every $0\neq a\in A$ there exists $\varepsilon>0$ such that: $$ \{k:\ \rank(\varphi_k(a))>\varepsilon n_k\} \in \mathcal{U}, $$ and for every $a,b\in A$ and $\varepsilon>0$, we have: $$ \{k:\ \rank\left(\varphi_k(a)\varphi_k(b)-\varphi_k(ab)\right)<\varepsilon n_k\}\in \mathcal{U}. $$
In particular, since $\mathcal{U}$ is a non-principal ultrafilter, we can fix a positive real $\alpha>0$ and a linear map $\varphi\colon A\rightarrow M_n(F)$ such that:
$$ \rank(\varphi(x)),\ \rank(\varphi(z))  > \alpha n. $$
By the assumptions of the lemma, $x=yxa$ and $z=xb$ for some $a,b\in A$, and $yz=0$. Let $T:=\varphi(x)-\varphi(y)\varphi(x)\varphi(a)$ and $S:=\varphi(z)-\varphi(x)\varphi(b)$. We may additionally assume that:
$$ \rank(\varphi(y)\varphi(z)),\ \rank(T),\ \rank(S) < \frac{\alpha}{4}n. $$


\textit{Claim}. We have: $$\rank(\varphi(y)\varphi(x)) <  \rank(\varphi(x))-\frac{\alpha}{2} n.$$
\textit{Proof of Claim}. 
First, since $\varphi(z)=\varphi(x)\varphi(b)+S$: $$ \Image(\varphi(z))\subseteq \Image(\varphi(x)) + \Image(S). $$
It follows that:
\begin{eqnarray*}
\frac{\Image(\varphi(z))}{\Image(\varphi(z)) \cap \Image(\varphi(x))} & \cong & \frac{\Image(\varphi(z))+\Image(\varphi(x))}{\Image(\varphi(x))} \\
& \subseteq & \frac{\Image(\varphi(x))+\Image(S)}{\Image(\varphi(x))} \\
& \cong & \frac{\Image(S)}{\Image(\varphi(x)) \cap \Image(S)}
\end{eqnarray*}
So: $$ \dim_F \left( \Image(\varphi(z))\cap \Image(\varphi(x))\right) \geq \rank(\varphi(z))-\rank(S)>\frac{3}{4}\alpha n. $$
Denote $\mathcal{V} := \Image(\varphi(y)\varphi(z))$ and recall that $\dim_F \mathcal{V} < \frac{\alpha}{4}n$.
Fix a direct sum complement of $\Image(\varphi(z))\cap \Image(\varphi(x))$ inside $\Image(\varphi(x))$, say, $\mathcal{W}$, and notice that: \begin{eqnarray*} \dim_F \mathcal{W} & = & \rank(\varphi(x)) - \dim_F \left( \Image(\varphi(z))\cap \Image(\varphi(x)) \right) \\ & < & \rank(\varphi(x)) - \frac{3\alpha}{4} n. \end{eqnarray*}
Now:
\begin{eqnarray*}
\rank(\varphi(y)\varphi(x)) & = & \dim_F \Image(\varphi(y)\varphi(x)) \\
& = & \dim_F \varphi(y)\Image(\varphi(x)) \\ & \leq & \dim_F \varphi(y)\left(\Image(\varphi(z)) \cap \Image(\varphi(x)) \right) + \dim_F \varphi(y)\mathcal{W} \\
& \leq & \dim_F \mathcal{V} + \dim_F \mathcal{W} \\
& < & \frac{\alpha}{4}n + \left(\rank(\varphi(x)) - \frac{3\alpha}{4} n \right) \\ & = & \rank(\varphi(x))-\frac{\alpha}{2} n.
\end{eqnarray*}

\smallskip
Return to the proof of the lemma.
Since $\varphi(x)=\varphi(y)\varphi(x)\varphi(a)+~T$ for some $T\in M_n(F)$ with $\rank(T)<\frac{\alpha}{4} n$, we have:
\begin{eqnarray*}
\Image \varphi(x) & = & \Image(\varphi(y)\varphi(x)\varphi(a)+T) \\
& \subseteq & \Image \left( \varphi(y) \varphi(x)\right) + \Image(T)
\end{eqnarray*}
whose dimension is at most $ \rank\left(\varphi(y)\varphi(x)\right) + \rank(T) $ which is, by the above claim, at most: $$ \rank(\varphi(x))-\frac{\alpha}{2} n + \rank(T) < \rank(\varphi(x))-\frac{\alpha}{4} n, $$
a contradiction. Hence $A$ is non linear sofic.
\end{proof}

\section{Stably finite non linear sofic algebras}

The following observation is well-known, and is brought here for the reader's convenience:

\begin{lem} \label{lem_jac}
Let $R$ be a ring and $J\triangleleft R$ its Jacobson radical. If $R/J$ is stably finite then so is $R$.
\end{lem}

\begin{proof}
Suppose that $X,Y\in M_n(R)$ satisfy $XY=I$. Since $R/J$ is stably finite then $I-YX\in M_n(J)$. Since $J$ is a quasi-invertible ideal, so is $M_n(J)\triangleleft M_n(R)$ and hence $YX=I-(I-YX)$ is invertible, so both $X,Y$ are invertible and since $XY=I$ it follows that $X=Y^{-1}$ and $YX=I$.
\end{proof}

\begin{proof}[{Proof of Theorem \ref{main}}]
Let $A=F\left< x,y \right> / \left<x^2, yxy-x \right>$. This algebra was introduced by Irving \cite{Irving} as an example of a finitely presented PI algebra which is not embeddable into any matrix algebra over a field.

The set of monomials in $x,y$ which avoid occurrences of $x^2$ and $yxy$ forms a linear basis for $A$; this fact was established in \cite{Irving}.
Indeed, this is a direct consequence of Bergman's Diamond Lemma \cite{Bergman}, since the only overlap between the reductions $yxy\mapsto x,\ x^2\mapsto 0$ is: $$ 0=xxy=(yxy)xy=yx(yxy)=yxx=0. $$
By \cite[Theorem~2]{Irving}, $A$ satisfies a polynomial identity (and has linear growth). Explicitly, since $\left<x\right>\triangleleft A$ satisfies $A/\left<x\right>\cong F[y]$ and $\left<x\right>^3=0$, the identity: $$ [X_1,Y_1][X_2,Y_2][X_3,Y_3]=0 $$ holds in $A$.

Any PI-ring is stably finite; this is well-known. Indeed, suppose that $R$ is a PI-ring and $J\triangleleft R$ is its Jacobson radical. Each primitive homomorphic image of $R$ is a simple algebra which is finite-dimensional over its center \cite{Kap}, hence stably finite by linear algebra. Thus, $R/J$ is a subdirect product of stably finite rings, so it is stably finite itself. By Lemma \ref{lem_jac}, $R$ itself is stably finite. In particular, our specific ring $A$ is stably finite. 

Finally, $A$ fulfills the requirements of Lemma \ref{main_lem} with $z=xyx$. Indeed, $xyx$ contains no occurrences of $x^2$ or $yxy$, and is thus non-zero; obviously, $xyx\in xA$; and finally, $$ yz=yxyx=x^2=0. $$ Hence $A$ is non linear sofic.
\end{proof}

\begin{proof}[{Proof of Theorem \ref{main_jac}}]
Let $A=F\left<\left<x,y\right>\right>$ be the ring of noncommutative formal power series and let $A^{+}$ be the ideal of $A$ consisting of all power series with zero constant term. Let $I=\left<yx^2y-x\right>\triangleleft A^{+}$. Since $A^{+}$ is Jacobson radical, the quotient ring $R:=A^{+}/I$ is also Jacobson radical. By \cite[\S 2, ``Basic Lemma'']{CohSas}, the image of $x$ in $R$ (for simplicity, we identify elements in $A^{+}$ with their images modulo $I$, by abuse of notation) is non-zero. Define a ring $S$ as follows. As an $F$-vector space, $$ S=R\oplus R^{1}z $$ (where $R^{1}$ is the unital hull of $R$) and multiplication is given by: $$ (s_0+s_1z)\cdot (t_0+t_1z) = s_0t_0+s_0t_1z $$
It is easy to see that $\left<z\right>^2=0$ and $S/\left<z\right>\cong R$, so $S$ is Jacobson radical. Consider $K=\left<yxz\right>\triangleleft S$. We claim that $xz\notin K$. Otherwise, write:
$$ (*)\ \ \ xz=\sum_{i=1}^{m} u_i(yxz)v_i $$
for some $u_i,v_i\in S^{1}$. Write $u_i=u_i^0+u_i^1z,v_i=v_i^0+v_i^1z$ where $u_i^0,v_i^0,u_i^1,v_i^1\in R^1$. Since $zs$ is a scalar multiplication of $z$ for an arbitrary $s\in S^1$ and $\left<z\right>^2=0$ we can rewrite $(*)$ as:
$$ xz=\sum_{i=1}^{m} \alpha_iu_i^0yxz $$
for some $\alpha_i\in F$, or equivalently:
$$ xz=fxz\ \ \ \text{where}\ \ \ f=\sum_{i=1}^{m} \alpha_iu_i^0y \in R $$
let $g$ be the quasi-inverse of $f$, namely, $gf=f+g$. Then $gxz=gfxz=fxz+gxz$, so $xz=fxz=0$. This contradicts that $x\neq 0$.

It follows that $xz\notin K$, so in the quotient ring $T:=S/K$ we have (again identifying elements in $S$ with their images modulo $K$): $$ xz\neq 0,\  yxz=0,\ x=yx^2y \in yxT $$
By Lemma \ref{main_lem} (in which the role of `$z$' is now taken by $xz$), the algebra $T$ is non linear sofic; since $S$ is Jacobson radical, so is $T=S/K$. The unital hull $T^{1}$ is non linear sofic and stably finite by Lemma \ref{lem_jac}.
\end{proof}

We conclude with:

\begin{ques}
Suppose that $J$ is a Jacobson radical algebra. If $J^{\circ}$ is a linear sofic group, must $J$ be a linear sofic algebra?
\end{ques}

If the answer to this question is affirmative then Theorem \ref{main_jac} gives an example for a non linear sofic (hence, non-sofic) group.

\end{document}